\documentclass[letterpaper, 10 pt, conference]{ieeeconf}  

\IEEEoverridecommandlockouts                              

\usepackage{amsmath} 
\usepackage{amssymb,bm}  
\usepackage{cite}
\usepackage[usenames]{color}
\usepackage{graphicx} 
\newtheorem{theorem}{Theorem}
\newtheorem{lemma}{Lemma}
\newtheorem{assumption}{Assumption}

\newtheorem{statement}{Statement}
\newtheorem{remark}{Remark}

\title{\LARGE \bf
    Extremum Seeking with High-Order Lie Bracket Approximations: Achieving  Exponential Decay Rate}

\author{Victoria Grushkovskaya$^{1,3}$ and Sameh A. Eisa$^{2}$
\thanks{$^{1}$Victoria Grushkovskaya is with the Department of Mathematics,
        University of Klagenfurt, 9020 Klagenfurt am W\"orthersee, Austria
        {\tt\small viktoriia.grushkovska@aau.at}}%
\thanks{$^{2}$Sameh A. Eisa is with the Department of Aerospace Engineering and Engineering Mechanics, University of Cincinnati, Ohio, United States {\tt\small eisash@ucmail.uc.edu}}
\thanks{$^{3}$Institute of Applied Mathematics \& Mechanics, National Academy of Sciences of Ukraine}
}

\begin{document}

\maketitle
\thispagestyle{empty}
\pagestyle{empty}

\begin{abstract}

This paper focuses on the further development of the Lie bracket approximation approach for extremum seeking systems. Classical results in this area  provide extremum seeking algorithms with exponential convergence rates 
for quadratic-like cost functions, and polynomial decay rates for cost functions of higher degrees. 
This paper proposes a novel control design approach that ensures the motion of the extremum seeking system along directions associated with higher-order Lie brackets, thereby ensuring exponential convergence for cost functions that are polynomial-like but with degree greater than two.

\end{abstract}
\section{Introduction }
Consider the following control-affine system:
\begin{equation}\label{eqn:controlAffineIntro}
    \dot{x}=f_0(x)+ \sum_{i=1}^{m} {\varepsilon^{-p_i}} f_i(x) u_i\left(k_i{t}/{\varepsilon}\right),\\
\end{equation}
where $x \in \mathbb{R}^n$ is the state space vector, $p_i\in (0,1)$, $0<\varepsilon\ll1$, $f_0$ is the drift (uncontrolled) vector field of the system, $f_i$ are the control vector fields, $u_i$ are the control inputs, $m  \in \mathbb N$ is the number of control inputs, and $k_i\in\mathbb Q$. Control-affine systems of the form~\eqref{eqn:controlAffineIntro} arise in many real-world applications, including (but not limited to) robotics, multi-agent systems, and flight dynamics (e.g., \cite{krstic2008extremum,grushkovskaya2018family,ECC_2024,eisa2023}). 
The nature of system~\eqref{eqn:controlAffineIntro} allows  applications of geometric control methods and techniques relying on Lie bracket approximation ideas\cite{kurzweil1987limit,Sus91,bullo04,Liu97,DurrAuto,Sch14,GZE18,Labar19,PE23}.
Such approaches have been widely used in the study of motion planning for underactuated systems, including those with nonholonomic constraints (e.g., \cite{bullo04,Liu97,Sus91,suttner2020extremum,ZuSIAM,doi:10.1080/00207179.2016.1257157,GZ23,GZ24_CDC}), as well as in model-free optimization and control via extremum seeking (ES) approaches \cite{DurrAuto,Sch14,GZE18,Labar19,PE23}, which is the main focus of this paper. 
%
Extremum seeking is  model-free, real-time dynamic optimization and control framework aimed at stabilizing a dynamical system about the extremum point of an objective function that is accessible through measurements but whose analytical form is unknown~\cite{Ar03,scheinker2024100}.
ES methods referred to as classic-based (e.g., \cite{Kr00,yilmaz2025exponential,guay2015time}) rely on classical averaging methods \cite{khalil2002nonlinear,Maggia2020higherOrderAvg} for analysis and design. Another family of approaches exploit control-affine ES systems and 
rely on Lie bracket approximation for analysis and design, such as, e.g.,  \cite{DurrAuto,Sch14,GZE18,Labar19,PE23}. 
%
In particular, for the case $p_i=0.5,k_i=1\text{ for all }i$, a first-order Lie bracket system (LBS) approximation of~\eqref{eqn:controlAffineIntro} and ES control laws were given, e.g., in \cite{DurrAuto,Sch14,GZE18}. 
Note that such methods define an ES system whose first-order LBS behaves like a gradient flow of the objective function. In \cite{Labar19}, the conditions on $p_i$ and $k_i$ were further generalized compared to~\cite{DurrAuto}, and a second-order LBS was introduced to approximate \eqref{eqn:controlAffineIntro}. This extension enabled the introduction of a Newton-based ES approach, since the inclusion of second-order Lie brackets provided access to  second-order derivative information (i.e., to the Hessian), and thus allowed to have a behavior similar to a Newton-type flow. In the recent paper \cite{PE23}, the authors proposed a generalized framework for constructing higher-order LBSs to approximate \eqref{eqn:controlAffineIntro}. They also showed that LBS approximations themselves are averaging terms, which guarantees  the closeness of trajectories between  LBSs and the original system \eqref{eqn:controlAffineIntro}, provided that $\varepsilon$ is  small enough. 
 Moreover, it was observed that ES designs based on third-order LBS possess a faster convergence rate when compared with even Newton-based ES; the authors in \cite{PE23} argued (without proof) that the observed faster convergence is due to the inclusion of third-order Lie brackets, which provide access to higher-order derivatives beyond Hessian (i.e., third-order derivative information).

\textbf{Motivation \& contributions.} Inspired by the concept of higher-order Lie bracket averaging, this paper further explores the application of these techniques to ES problems, with a particular focus on achieving faster convergence rates. In particular, many ES algorithms based on first-order Lie bracket approximation exhibit exponential convergence when the cost function  $J(x)$ behaves locally like a quadratic function near the extremum point $x^*$  (i.e., $J(x)\sim \|x-x^*\|^{2}$). However, if the cost function behaves like a higher-degree polynomial near the extremum, i.e., $ J(x) \sim \|x - x^*\|^m  $ with $ m > 2$, such algorithms exhibit only a polynomial decay rate~\cite{GZE18}.
 Even though it was observed in \cite{PE23} that using higher-order Lie brackets to design ES may lead to faster convergence rate when $J(x)\sim \|x-x^*\|^{4}$, no conclusion or a proof was provided regarding the nature of the faster convergence resulting from higher-order Lie brackets. 
 
 In this paper, we provide a preliminary study of a class of ES problems related to the unconstrained minimization of a cost function $J(x)$.
The function $J$ is unknown in terms of an explicit analytic expression, but it can be evaluated (measured) at any point. We focus on designing an ES control system of the form
$
 \dot x = u(t,J(x)),
$
such that the system's trajectories tend  to an extremum point of $J$. To develop our approach,  we assume that the cost function  behaves locally like an $m$-th order polynomial near the minimizer $x^*\in\mathbb R^n$, i.e., $J(x)\sim \|x-x^*\|^{m}$, with some $m\ge 2$. 

The main contribution of this paper is a novel ES design framework that leverages the excitation of higher-order Lie brackets to steer the system along directions corresponding to higher-order derivatives of the cost function. This approach helps to increase the convergence rate, in particular,  ensuring the exponential convergence even for cost functions that are not quadratic in nature. Furthermore, we generalize the result of~\cite{GZE18}, which described a family of vector fields whose first-order Lie bracket equals the gradient of the cost function. In this paper, we extend this idea by deriving a formula that generates vector fields such that the corresponding $\ell$-th order Lie bracket equals the $\ell$-th derivative of the cost function. In addition, we discuss several strategies for exciting Lie brackets using different numbers of dither signals, which gives  higher flexibility in the design of control vector fields. 

\section{Preliminaries}

\subsection{Notations} 

Throughout the text, $\mathbb R^+=[0,\infty)$  denotes the set of all non-negative real numbers;
 $B_\delta(x^*)$ is the $\delta$-neighborhood of $x^*{\in} \mathbb R^n$,  $\overline{B_\delta(x^*)}$ is its closure;
 for  $h{\in} C^1(\mathbb R^n;\mathbb R)$, $\xi{\in}\mathbb R^n$, we denote 
 $\nabla h(\xi):=\frac{\partial h(x)}{\partial x}^T\Big|_{x=\xi}$ to be a column vector;
for $h\in C^N(\mathbb R;\mathbb R)$, we define its $\ell$-th order derivative as $h^{(\ell)}(x):=\frac{d^\ell h(x)}{dx^\ell}$, for any $\ell\in\{0,\dots,N\}$, with $h^{(0)}(x):=h(x)$;
for an $f:\mathbb R\to\mathbb R$,  $f(z)=O(z)$ as $z\to 0$ means that there is a $c>0$ such that $|f(z)|\le c|z|$ in a neighborhood of $0$;
for  $f,g:\mathbb R^n\to\mathbb R^n $, $x\in\mathbb R^n$, the Lie derivative is defined as
 $L_gf(x)=\lim _{s\to0}\frac{f(x+sg(x))-f(x)}{s}$, 
 and 
 $[f,g](x)=L_fg(x)-L_gf(x)$ is the first order Lie bracket. 
 To define higher-order Lie brackets, we introduce $\ell$-dimensional multi-index $I_\ell=(i_1,\dots ,i_\ell)$; then
 $f_{I_\ell}(x)= \left[\big[[f_{i_1},f_{i_2}],f_{i_3}\big],\dots,f_{i_\ell}\right](x)$ --  the right-iterated Lie bracket of length $\ell$, or $(\ell-1)$-order Lie bracket;
for $I_1{\in}\{1,...,m\}$, $f_{I_1}$ denotes a corresponding  vector field;
$L^\infty_{[0,\varepsilon]}$ -- the class of essentially  bounded measurable functions on $[0,\varepsilon]$.

\subsection{Lie bracket approximations} 
Consider the control-affine system \eqref{eqn:controlAffineIntro}, where $x=(x_1,\dots,x_n)^T{\in}\mathbb R^n$ is the state vector, $x(t_0)=x^0{\in}\mathbb R^n$ (without loss of generality, we assume $t_0=0$), $\varepsilon>0$ is a small parameter, $f_i{:\mathbb R^n\to\mathbb R^n}$ are continuously differentiable (up to any order) vector fields , and $u_i$ are continuous in $t$ and $T$-periodic functions with zero average, i.e., $\int _0^Tu_i(\tau)d\tau=0$. LBS approximations of \eqref{eqn:controlAffineIntro} up to a third-order are \cite{PE23}:
\begin{equation}\label{eqn:LBS_main}
    \dot {\bar x}{=}f_0(\bar x)+ \sum_{i = 1}^r L_i(\bar x),
    \end{equation}
with $ L_1 = 0$, $ L_2 = \sum_{\underset{j_2=j_1+1}{j_1 = 1}}^m [f_{j_1}, f_{j_2}]$,\\
    $    L_3 = \sum_{\underset{j_2=j_1+1}{j_1,j_3 = 1}}^m  \nu_{j_1 j_2 j_3} [f_{j_3},[f_{j_1},f_{j_2}]]$,\\
     $   L_4 = \sum_{{\underset{j_2=j_1+1}{j_1,j_3 = 1}}}^m\big( \sum _{j_4 = 1}^m \beta_{1_{j_1 j_2 j_3 j_4}} \Big[[f_{j_1},f_{j_2}],[f_{j_3},f_{j_4}]\Big]
        + \sum _{j_4 = j_3+1}^m\beta_{2_{j_1 j_2 j_3 j_4}} [[[f_{j_1},f_{j_2}],f_{j_3}],f_{j_4}]\big)$,
where $\nu_{j_1 j_2}$, $\nu_{j_1 j_2 j_3}$, $\beta_{1_{j_1 j_2 j_3 j_4}}$ and $\beta_{2_{j_1 j_2 j_3 j_4}}$ are coefficients resulting from the iterated integrals of the dither input signals (formulas are provided in \cite[Section 4]{PE23}). 
Truncating \eqref{eqn:LBS_main} at $r=2$ provides first-order LBS (e.g., \cite{DurrAuto,GZE18}). Similarly, truncating \eqref{eqn:LBS_main} at $r=3$ and $r=4$ provides second- and third-order LBS, respectively.
   The stability properties of systems~\eqref{eqn:controlAffineIntro} and \eqref{eqn:LBS_main} are related as follows. 
  \begin{lemma}[\cite{DurrAuto,GZE18,PE23}]~\label{dthm}
\textit{If a compact set $ S\subset\mathbb R^n$ is locally (globally) uniformly asymptotically stable for~\eqref{eqn:LBS_main} then it is locally (semi-globally) practically
uniformly asymptotically stable for~\eqref{eqn:controlAffineIntro}.
}  \end{lemma}	
  We omit the notions of practical asymptotic stability and practical exponential stability because of the space limits and refer the readers to, e.g.,~\cite{DurrAuto,GZE18}.

Lemma~\ref{dthm} establishes the relation between the solutions of control-affine system~\eqref{eqn:controlAffineIntro} and the corresponding 
first-order Lie bracket system
LBS~\eqref{eqn:LBS_main}. 
Let us recap the approach from \cite{GZE18} given its relevance to the contributions of this paper. For a special class of \eqref{eqn:controlAffineIntro} when $f_0=0$, $p_i=0.5$, $k_i=1$ for all $i$, the result of Lemma~\ref{dthm} can be exploited for solving the extremum seeking problem in the following way~\cite{DurrAuto,GZE18}: let us consider a class of ES systems of the form
\begin{equation}\label{int}
\dot x=\varepsilon^{-1/2}\Big(F_1(J(x))u_{1}^\varepsilon( t)+F_{2}(J(x)) u_{2}^\varepsilon( t)\Big),
\end{equation}
where  $F_1\circ J,F_2\circ J\in C^2(D)$ satisfy   $[F_1,F_2](Z)=1$ for all $z\in\mathbb R$ (which implies $[F_1\circ J,F_2\circ J](x)=\nabla J(x)$), $u_1^\varepsilon(t)=2\sqrt{{\pi}{\varepsilon^{-1}}}\cos\big({2\pi  t}{\varepsilon^{-1}}\big)$, $u_2^\varepsilon(t)=2\sqrt{{\pi}{\varepsilon^{-1}}}\sin\big({2\pi  t}{\varepsilon^{-1}}\big)$, and assume that the cost function $J$ satisfies the following properties in some domain $D\subseteq\mathbb R^n$:
\begin{assumption}
The function $J\in C^2(D,\mathbb R^n)$, and
\begin{itemize}
    \item [A1.1)] there is an $x^*{\in} D$  s. t. 
    $\nabla J(x){=}0 \text{ if and only if }x=x^*,$
    and $J(x^*)=J^*{\in}\mathbb R$, $J(x)>J(x^*)$ for all $x\in D{\setminus}\{x^*\}$;
    \item[A1.2)] there are $\alpha_1,\alpha_2,\beta_1,\beta_2,\mu$, and $m\ge 1$, such that
     $$ \begin{aligned}
     \alpha_1\|x{-}x^*\|^{m} \le &J(x)-J^* \le \alpha_2\|x{-}x^* \|^{m},\\
      \beta_1 (J(x)-J^*)^{1{-}\frac{1}{m}}\le  & \|\nabla J(x)\|\le\beta_2 (J(x)-J^*)^{1{-}\frac{1}{m}},\\
    \left\|\frac{\partial^2 J(x)}{\partial x^2}\right\|{\le}&\mu (J(x)-J^*)^{1{-}\frac{2}{m}}\text{ for all }x\in D.
    \end{aligned}
    $$
\end{itemize}
Assumption A1.1) states that the cost function $J$ possesses an isolated local minimum at $x^*$,where it attains the value  $J^*$.
  Assumption A1.2) reflects the requirement that $J$ exhibits a local behavior similar to that of a power function.
\end{assumption}
As follows from~\cite{GZE18}, under the above conditions the point $x=x^*$ is practically asymptotically stable for~\eqref{int}, with the convergence rate dependent on the parameter $m$ in A1.2): 
\begin{lemma}\label{lem_decay}
\textit{  If the cost function $J\in C^2(\mathbb R^n;\mathbb R)$ satisfies  Assumption~1 in a domain $D\subset \mathbb R^n$, then $x^*$ is
 practically exponentially stable for system~\eqref{int} if $m=2$, and $x^*$ is  practically asymptotically stable for system~\eqref{int} if $m>2$.
Namely, for any $\delta$ such that $\overline{B_\delta(x^*)}\subset D$,  any $\bar\lambda\in(0,\alpha\kappa_1)$, and $\rho\in(0,\delta)$, there exists an $\bar\varepsilon>0$ such that, for any $\varepsilon\in(0,\bar\varepsilon]$,
$\lambda\in(0,\bar\lambda]$, the solutions of system~\eqref{int} with $x^0{\in} B_{\delta}(x^*)$ exhibit the following decay rate:
\begin{itemize}
    \item if $m=2$, then
    $\|x(t)-x^*\|\le c_{m}\|x^0-x^*\|e^{-\lambda t}+\rho;$
    \item if $m>2$, then\\
    $
    \|x(t)-x^*\|\le \Big(c_{m1}\|x^0-x^*\|^{2-m}+ c_{m2} t\Big)^{-1/(m-2)}+\rho,
    $
    with some $c_m,c_{m1},c_{m2}>0$.
\end{itemize}
}
\end{lemma}
More technical details on the above decay rate estimates can be found in~\cite{GZE18}.
 For the sake of clarity, we have assumed here $x\in\mathbb R$, however, the above result can be easily extended to the case $x\in\mathbb R^n$~(see, e.g.,~\cite{DurrAuto,GZE18}).

\subsection{Main idea}

As follows from Lemma~\ref{lem_decay}, the Lie bracket approximation approach provides a constructive solution to the extremum seeking problem which ensures the exponential convergence of the trajectories of system~\eqref{int} to the optimal point in the case of a quadratic-like cost function, i.e., for $m=2$ in Assumption A2). However, for cases $m>2$, the above algorithm ensures only a polynomial decay rate. This can also be observed by analyzing the first-order Lie bracket system associated with system~\eqref{int}:
$
\dot{\bar x}=-\nabla J(\bar x).
$
Note that the above first-order LBS is a complete average asymptote for \eqref{int}, meaning that higher-order LBS approximations will be redundant \cite{PE23}. Now, if for example $J=\tfrac{1}{2}(x-x^*)^{2}$, then the Lie bracket system is linear, $\dot{\bar x}=-(\bar x-x^*) $, and thus $x^*$ is its exponentially stable equilibrium point. In case $J=\tfrac{1}{4}(x-x^*)^{4}$, the Lie bracket system takes the form $\dot{\bar x}=-(\bar x-x^*)^3$, and its solutions are well-known to exhibit the polynomial decay rate $O(t^{-1/2})$ as $t\to\infty$~\cite{GZ13,G16}.  Assume now that we can associate the properties of system~\eqref{eqn:controlAffineIntro} with a system which gives access to the third-order derivative of $J$, namely, with the system 
$\dot{\tilde x}=-J^{(3)}(\tilde x)=-6(\tilde x-x^*)$. Then the latter system turns out to be linear again, which, under certain assumptions, may imply the practical exponential stability of the extremum seeking system. A natural way of accessing higher order derivatives of the cost function is to excite the Lie brackets of corresponding order~\cite{Labar19,PE23}. For example, it is easy to see that
$
[[[J,1],1],1](x)=-J^{(3)}(x).
$
The main idea of this paper is to construct an extremum seeking system 
\begin{equation}
    \label{int_gen}
    \dot x=\sum_{k=1}^{n_u} g_k(J(x))u_k^\varepsilon(t),
\end{equation}
so that, under a special choice of  vector fields $g_k$, 
dither signals $u_k^\varepsilon{\in }L^\infty_{[0,\varepsilon]}$, and an $\varepsilon{>}0$,
 its trajectories approximate the trajectories of a system with high order Lie brackets:
 \begin{equation}
     \label{Lie_gen}
     \dot{\bar x}=\sum_{\ell=1}^N\sum_{I_\ell\in S_\ell}c_{I_\ell}g_{I_\ell}(\bar x),
 \end{equation}
 where  $N\in\mathbb N$,  $S_\ell\subset\{1,\dots,n_u\}^\ell$ denotes the sets of multi-indices of the Lie brackets required for solving theES problem, $g_{I_\ell}$ are the corresponding Lie brackets, and $c_{I_\ell}$ are constant parameters. For example, for system~\eqref{eqn:LBS_main} with $r=2$, $f_0=0$, $f_i=g_i$, $i\in\{1,\dots,n\}$, we mean $N=2$, $S_1=\emptyset$,   $S_2=\{(i,j):1\le i<j\le n\}$, $I_2:=(i,j)\in S_2$ $c_{I_2}=\nu_{ij}$.
 
One of the main tools exploited in this paper is the Chen--Fliess series expansion~\cite{La95,GZ23}: under certain regularity assumptions on the control vector fields of system~\eqref{int_gen}, the solutions of~\eqref{int_gen} with $x(0)=x^0$ can be represented as
\vspace{-0.25em}
{\small \begin{equation}
    \label{series_gen}
    \begin{aligned}
       &x(t)=x^0+\sum_{\ell=1}^N\sum_{k_1,\dots, k_\ell=1}^{n_u}L_{g_{k_\ell}}\dots L_{g_{k_2}}f_{g_1}(x^0) \\
       &\times\int _0^t\int _0^{s_1}... \int _0^{s_{\ell-1}}u_{k_1}^\varepsilon(s_1) u_{k_\ell}^\varepsilon(s_\ell) ds_\ell... ds_1 + R(t),
    \end{aligned}
\end{equation}}
with the remainder
{\small $$
\begin{aligned}
      R(t)=&\hspace{-1em}\sum_{k_1,\dots, k_{N+1}=1}^{n_u}\int _0^t\int _0^{s_1}... \int _0^{s_{N}}L_{g_{k_{N+1}}}\dots L_{g_{k_2}}g_{k_1}(x(s_{N+1})) \\
       &\times u_{k_1}^\varepsilon(s_1)\times \dots  \times u_{k_{N+1}}^\varepsilon(s_{N+1}) ds_{N+1}... ds_1.
    \end{aligned}
$$}
 
\section{Main results}
\subsection{Two-input extremum seeking system}

To simplify the presentation, in this section, we assume $x\in \mathbb R$. To steer the solutions of an extremum seeking system towards the direction of high-order Lie brackets, we refer to control approaches from nonholonomic systems theory~\cite{Sus91,Liu97,ZuSIAM,Gau14,GZ23,GZ24_CDC}. We focus here on the two-input systems of form~\eqref{int_gen}: 
\begin{equation}\label{int_2}
\dot x=g_1(J(x))u_{1}^\varepsilon(t)+g_{2}(J(x)) u_{2}^\varepsilon(t).
\end{equation}
Suppose also  that $u_1^\varepsilon(t)=\varepsilon^{1/N-1}v_1(t/\varepsilon)$, $u_2^\varepsilon(t)={\varepsilon^{1/N-1}}v_2(t/\varepsilon)$, with the dithers $v_1(t/\varepsilon),v_2(t/\varepsilon)$ 
exciting the Lie bracket 
$
g_{I_N}(z)=\big[\big[\dots[g_1,\underbrace{g_2],g_2\big],\dots,g_2}_{N-1\text{ times}}]\dots]\big](z)
$
at time $t=\varepsilon$, with $I_N=(1,\underbrace{2,\dots,2}_{N-1\text{ times}})$, in the sense that the Chen--Fliess series expansion~\eqref{series_gen}  takes the form
\begin{equation}
    \label{series_2}
    x(\varepsilon)=x^0+\varepsilon g_{I_N}(J(x^0))+R(\varepsilon),
\end{equation}
and all the other Lie brackets of length from 1 to $N$ do not appear in the above expansion. 
One way to construct such inputs is described in~\cite{Gau14}, other approaches can be found in, e.g.,~\cite{Sus91,Liu97,PE23}. The concrete examples are as follows, as mentioned in~\cite{GZ24_CDC}.
 \begin{statement} 
 \emph{   \begin{itemize}
    \item the inputs
    $v_1(t/\varepsilon)=2\sqrt{\kappa_{12}\pi}\cos\left({2\kappa_{12}\pi t}/{\varepsilon}\right)$, \\
    $v_2(t/\varepsilon)=2\sqrt{{\kappa_{12}\pi}}\sin\left({2\kappa_{12}\pi t}/{\varepsilon}\right)$, $\kappa_{12}\in\mathbb Z$,\\
    excite the first order Lie bracket $[g_1,g_2]$;
    \item the inputs 
    $v_1(t/\varepsilon)=-2\left({4\kappa_{122}\pi }\right)^{\frac23}\cos\left({4\kappa_{122}\pi t}/{\varepsilon}\right)$,\\
      $v_2(t/\varepsilon)=\left({4\kappa_{122}\pi }\right)^{\frac23} \cos\left({2\kappa_{122}\pi t}/{\varepsilon}\right)$,\\ $\kappa_{112}\in\mathbb Z$,
         excite the second order Lie bracket $[[g_1,g_2],g_2]$;
      \item  the inputs  
      $v_1(t/\varepsilon)=6\left(2\kappa_{1222}\pi\right)^{\frac34}\sin\left({6\kappa_{1222}\pi t}/{\varepsilon}\right)$,\\
      $v_2(t/\varepsilon)= 2\left(2\kappa_{1222}\pi\right)^{\frac34}\cos\left({2\kappa_{1222}\pi t}/{\varepsilon}\right)$, \\$\kappa_{1222}\in\mathbb Z$,
         excite the third order Lie bracket $\big[[[g_1,g_2],g_2],g_2\big]$.
\end{itemize}}
\end{statement}

Assume further that $g_1,g_2$ are chosen in such a way that the Lie bracket $g_{I_N}(J(x))$ has the form
\begin{equation}\label{g_choice}
g_{I_N}(J(x))=-c_NJ^{(N-1)}(x),
\end{equation}
 and some $c_N>0$ playing a role of control gain parameter. 
In this context, $g_{I_N}(J(x))$ denotes the corresponding Lie bracket computed with the compositions of functions $g_1\circ J$, $g_2\circ J$.
The most obvious choice of the vector fields satisfying relation~\eqref{g_choice} is $g_1(z)=(-1)^{N+1}z$, $g_2(z)\equiv 1$. For the first-order Lie brackets, the whole family of functions $g_1,g_2$ satisfying the relation
\begin{equation}
    \label{g_choice_2}
    [g_1\circ J,g_2\circ J](x)=-\varphi(J(x))\nabla J(x),
\end{equation}
with any given continuous  function $\varphi$,
has been introduced in~\cite{GZE18}.  Then the expansion~\eqref{series_2} takes the form
\begin{equation}
    \label{series_2eps}
    x(\varepsilon)=x^0-\varepsilon c_NJ^{(N-1)}(x^0)+R(\varepsilon), 
\end{equation}
and, similarly to the approach of~\cite{GZE18}, the practical asymptotic stability of $x^*$ for system~\eqref{int_2} can be proved. We proceed by summarizing the key results of this subsection and integrating them into  the context of solving the extremum seeking problem. For this purpose, we further  specify the properties of  $J$ as a polynomial-like single-variable function:
\begin{assumption}
    The function $J\in C^m(D,\mathbb R^{m})$ with some $m\ge 2$, and
there exist constants  $\alpha_{1},\alpha_{2},\beta_{1},\beta_{2}$, such that, for all $x\in D$,
     $$ 
     \begin{aligned}
    \alpha_{1}\|x-x^*\|^{m} \le  J(x)-J^* & \le \alpha_{2}\|x-x^* \|^{m},\\
    J^{(m-1)}(x)(x-x^*)&\ge\beta_{1}\|x^0-x^*\|^2,\\  \|J^{(m-1)}(x)\|&\le \beta_{2}\|x^0-x^*\|.
    \end{aligned}
    $$
\end{assumption}

\begin{theorem}
    \label{thm_2inp}
\textit{Given system~\eqref{int_2} with $N=m$ and a cost function $J$ satisfying Assumption~2, let the vector fields $g_1\circ J,g_2\circ J\in C^m(D;\mathbb R^n)$ satisfy the relation~\eqref{g_choice} with $I_m=(1,{2,\dots,2})$, and let $u_1^\varepsilon,u_2^\varepsilon\in L^\infty_{[0,\varepsilon]}$ be $\varepsilon$-periodic dithers which ensure the representation~\eqref{series_2}. Then the point $x^*$ is semi-globally practically exponentially stable for system~\eqref{int_2}.}
\end{theorem}
\begin{proof}
  The argumentation is similar to  the proof of the practical exponential stability in~\cite[Theorem~3]{GZE18}, so we only explain here how to derive exponential decay rate to a neighborhood of $x^*$ based on the representation~\eqref{series_2}.   Given any $\delta>0$, let $D'$ be any compact set such that $\overline{B_\delta(x^*)}\subset D'\subset D$. Similarly to the proof of~\cite[Theorem~3]{GZE18}, one can show that there exists an $\varepsilon_0>0$ such that solutions of system~\eqref{int_2} with $x^0\in B_\delta(x^*)$ and $\varepsilon\in(0,\varepsilon_0)$ are well-defined in $D'$ for all $t\in[0,\varepsilon]$. To estimate the remainder in~\eqref{series_2eps}, denote 
  $c_{Fm}=\sup _{x\in D'}\sum_{k_1,\dots, k_{m+1}=1}^{2}L_{g_{k_{m}}}\dots L_{g_{k_2}}g_{k_1}\circ J(x),$ $c_u=\max _{k=1,2}\sup _{t\in[0,\varepsilon]}|v_k(t/\varepsilon)|$. Then it is easy to see that $\|R(\varepsilon)\|\le c_R\varepsilon^{1+\frac1m}$ with
   $c_R= c_u^{m+1}c_{Fm}$.
The above estimate together with the representation~\eqref{series_2eps}, Assumption~2, and triangular inequality, implies that
$$
\begin{aligned}
 \|x(\varepsilon)-x^*\|^2\le &\|x^0-x^*\|^2(1-2\varepsilon \gamma_m)+\varepsilon^{1+\frac1m}\sigma,\\
\end{aligned}
$$
with $\gamma_m=c_m(\beta_{1}-\varepsilon c_m \beta_{2}^2)$, $\sigma= c_R^2\varepsilon^{1+\frac1m}+2c_R\delta(1+\varepsilon  c_m\beta_{2})$. Then following the argumentation of~\cite[Steps 3.I-4.I]{GZE18}, we can establish the practical exponential decay rate to an arbitrary small neighborhood of $x^*$.
\end{proof}
\begin{remark}
As in the paper~\cite{GZE18}, it is also possible to relax regularity assumption on the control vector fields. Namely, requirement of  $g_k\circ J$, $k=1,2$, being $m$ times continuously differentiable in $D$ can be replaced with the following: $g_k\circ J\in C^m(D\setminus\{x^*\});\mathbb R$ and  $L_{g_{k_{N+1}}}\dots L_{g_{k_2}}g_{k_1}\circ J\in C(D;\mathbb R)$ for all $N\in\{1,\dots,m\}$, $k_1,\dots,k_{m+1}\in\{1,2\}$. This relaxation is particularly important for deriving conditions for the ``classical'' exponential stability in the sense of Lyapunov, meaning that the trajectories of the extremum seeking system converge to the point
 $x^*$, rather than merely to its neighborhood. Another important condition for achieving classical exponential stability is the property of vanishing amplitudes, which requires that $g_k\circ J\to 0$  as $x\to x^*$ (see~\cite[Theorem~3, Part II]{GZE18}). Since selecting  such vector fields becomes increasingly challenging in the case of higher-order Lie brackets, we leave this task, along with a rigorous formulation of the corresponding exponential stability properties, for further research.
\end{remark}
\subsection{Alternative design approaches}     
In the previous subsection, we have described the approach for generating extremum seeking systems with two inputs (dithers). However, this method for realizing the dynamics similar to $\dot{\bar x}=-J^{(m-1)}$ is clearly not unique. For example, an alternative way to excite a Lie bracket of length 2 is to introduce a three-input strategy
\begin{equation}\label{int_3}
\dot x=g_1(J(x))u_{1}^\varepsilon(t)+g_{2}(J(x)) u_{2}^\varepsilon(t)+g_{3}(J(x)) u_{3}^\varepsilon(t),
\end{equation}
 with $u_k^\varepsilon(t)=\varepsilon^{-3/4}v_k^\varepsilon(t)$, $k=1,2,3$,  $v_k(t/\varepsilon)$ 
exciting the Lie bracket $g_{(1,2,3)}=[[g_1,g_2],g_3]$, and $g_1,g_2,g_3$ such that
\begin{equation}
    \label{g_choice_3}
    \big[[g_1\circ J,g_2\circ J],g_3\circ J\big](x)=-c_{(1,2,3)}J^{(2)}(x)
\end{equation}
with some $c_{(1,2,3)}>0$. The reason for introducing more inputs in an extremum seeking system is to gain more flexibility in selecting the control vector fields. 
As in previous subsection, formula~\eqref{g_choice_3} holds for $g_1(z)=-z$, $g_2(z)=1$, $g_3(z)=1$. The three-input structure, however, facilitates a more general description of the entire class of functions  $g_1,g_2,g_3$ satisfying~\eqref{g_choice_3}. Namely, the paper~\cite{GZE18} provides a general formula for deriving $g_1,g_2$ such that the relation~\eqref{g_choice_2} is satisfied with an arbitrary  $\varphi:\mathbb R\to\mathbb R$. Then straightforward computations yields the following result.
\begin{lemma}
\emph{Let $\varphi_2\in C^1(\mathbb R;\mathbb R)$ be any given function,   the functions $g_1,g_2\in C^1(\mathbb R;\mathbb R)$ satisfy~\eqref{g_choice_2} with   $\varphi=\varphi_2$, and let $g_3=-\varphi_2$. Then, for all $x\in D$,
$$
\big[[g_1\circ J,g_2\circ J],g_3\circ J\big](x)=-\varphi_2^2((J(x)))J^{(2)}(x).
$$}
\end{lemma}
In particular, by setting $\varphi_2(z)\equiv \sqrt{c_{(1,2,3)}}$ in the above Lemma, we  directly obtain formula~\eqref{g_choice_3}. 

As noted in Remark~1, the regularity assumptions on $g_1,g_2,g_3$ can be relaxed, and with appropriate selections, it is possible to achieve exponential stability in the  sense of Lyapunov. A formal statement of this result is left for future work. 

\begin{remark}
    With the exception of special cases (e.g., Newton-based ES as in \cite{Labar19}), in general, the even-order derivatives are rather not helpful for classic extremum seeking problems. 
The goal of considering  this case here is to illustrate alternative approaches to designing extremum seeking systems, both in terms of selecting control vector fields and determining the number of dithers.  The formula presented in Lemma~3 can be easily extended to higher-order scenarios. For example, given a four-input system~\eqref{int_gen} with $n_u=4$ and any   function $\varphi_3\in C^1(\mathbb R;\mathbb R^+)$ such that $\sqrt{\varphi_3}\in C^1(\mathbb R^+;\mathbb R^+)$, let the functions $g_1,g_2,g_3$ be  chosen as in Lemma~3 with $\varphi_2=\sqrt{\varphi_3}$. Then $\big[\big[[g_1\circ J,g_2\circ J],g_3\circ J\big],g_4\circ J\big](x)=-\varphi_3^2J^{(3)}(x)$.
\end{remark}
To conclude this subsection, we provide a generalized formulation of Theorem~\ref{thm_2inp}.
\begin{theorem}
    \label{thm_minp}
\textit{Given system~\eqref{int_2} with $N=m$ and a cost function $J$ satisfying Assumption~2, let the vector fields $g_1\circ J,\dots,g_{m}\circ J\in C^m(D;\mathbb R^n)$ satisfy the relation~\eqref{g_choice} with $I_m=(i_1,\dots,i_m)\in\{1,\dots,m\}^m$, and let $u_1^\varepsilon,\dots,u_m^\varepsilon\in L^\infty_{[0,\varepsilon]}$ be $\varepsilon$-periodic dithers which ensure the representation~\eqref{series_2}. Then the point $x^*$ is semi-globally practically exponentially stable for system~\eqref{int_2}.}
\end{theorem}
The proof mostly repeats the proof of Theorem~\ref{thm_2inp}, additional details will be provided in the extended version of the paper.

\subsection{Polynomial-like cost functions with unknown degree}

Ensuring dynamics of the form \( \dot{\bar{x}} = -J^{(N-1)} \) is particularly useful when the degree \( m \) characterizing the behavior of the cost function is known. However, in extremum seeking problems, only very limited information about the cost function is typically available. In such cases, an incorrect choice of the order \( N \) may lead to undesirable system behavior, making it impossible to steer system to the desired state.
For example, consider the cost function \( J(x) = \frac{1}{2}(x - x^*)^2 \), and assume that inputs exciting the third-order Lie bracket are used in~\eqref{int_2}. The resulting third-order Lie bracket system then takes the form \( \dot{\bar{x}} = 0 \), which loses the desired stability properties. Correspondingly, the expansion~\eqref{series_2eps} becomes 
$x(\varepsilon) = x^0 + R(\varepsilon),$
which also demonstrates the ineffectiveness of such control in solving the extremum seeking problem in this case.

To overcome this limitation, we employ a splitting of the time-varying dithers, where different dithers are assigned to excite Lie brackets of different orders. A clever choice of dither's frequencies allows  to obtain an associated Lie bracket system of the form
$
\dot{\bar x}=-\sum_{j=1}^{N}\gamma_jJ^{(j)}(\bar x)\text{ with some } \gamma_j\ge0,
$
which, under appropriate  assumptions on $J$, possesses local exponential stability in a neighborhood of $x^*$. 
To illustrate the method clearly and stay within page limits, we excite each Lie bracket with two corresponding inputs, and restrict our study to $ N = 3 $.
\begin{assumption}
     The function $J\in C^m(D,\mathbb R^{m})$ with some $m\ge 2$, and
there exist constants  $\alpha_{1},\alpha_{2}>0$,  $\beta_{\ell 1},\beta_{\ell 2}\ge0$, $\ell=1,2$, such that, for all $x\in D$,
     $$ 
     \begin{aligned}
    &\alpha_{m1}\|x-x^*\|^{m} \le  J(x)-J^*  \le \alpha_{m2}\|x-x^* \|^{m},\\
    &J^{(1)}(x)(x-x^*)\ge\beta_{11}\|x^0-x^*\|^{m},\\
  &\|J^{(1)}(x)\|\le \beta_{12}\|x^0-x^*\|^{m-1},\\
   &J^{(3)}(x)(x-x^*)\ge\beta_{21}\|x^0-x^*\|^{m-2},\,  \|J^{(3)}(x)\|\le \beta_{22},
    \end{aligned}
    $$
assuming $\beta_{21}=0$ if $m=2$ and $\beta_{21}\ge 0$ otherwise. 
\end{assumption}
\begin{theorem}
\textit{Given a cost function satisfying Assumption~3,  let the extremum seeking system have the form 
\begin{equation}
    \label{int_many}
    \dot x= \sum_{k=1}^4 g_k(J(x))u_k^\varepsilon(t),
\end{equation}
where 
 $u_1^\varepsilon(t)=\varepsilon^{-1/2}v_1^{12}(t/\varepsilon)$, $u_2^\varepsilon(t)=\varepsilon^{-1/2}v_2^{12}(t/\varepsilon)$, 
$ u_3^\varepsilon(t)=\varepsilon^{-3/4}v_1^{1222}(t/\varepsilon)$,  $u_4^\varepsilon(t)=\varepsilon^{-3/4}v_2^{1222}(t/\varepsilon)$,
with the dithers $v_j^{12}, v_j^{1222}$, $j=1,2$, chosen as in Statement~1 under assumption that  there are no resonances of order up to $4$ in each of the following pairs: $(\kappa_{12}, \kappa_{1222})$,  $(\kappa_{12}, 3\kappa_{1222})$, and $(\kappa_{12}, 3\kappa_{1222})$. 
\\
Further, assume that the functions $g_k\circ J\in C^4(D;\mathbb R)$ satisfy the following relations:
$  [g_1\circ J,g_2\circ J](x)=-\gamma_1 J^{(1)}(x)$,
$   \big[\big[[g_3\circ J,g_4\circ J],g_4\circ J\big],g_4\circ J\big](x)=-\gamma_3  J^{(3)}(x)$.
Then the point $x^*$ is practically exponentially stable for system~\eqref{int_many} if $m=2$ or $m=4$, and practically asymptotically stable for system~\eqref{int_many} if $m>4$.}
\end{theorem}
We omit the proof because of the space limits. It is similar to the proof of Theorem~1 and uses the fact that, as it is shown in~\cite{GZ24_CDC}, the non-resonance assumption implies that the Chen--Fliess expansion~\eqref{series_gen} for system~\eqref{int_many} takes the form
$$
x(\varepsilon)=x^0- \gamma_1J^{(1)}(x^0)- \gamma_2J^{(3)}(x^0)+R(\varepsilon),
$$
with $\|R(\varepsilon)\|\le C_R\varepsilon^{5/4}$ with some $C_R\ge 0$, and thus, 
$$
\begin{aligned}
\| &x(\varepsilon)-x^0\|^2=\|x^0-x^*\|^2\\
&-2\varepsilon(\gamma_1J^{(1)}(x^0)+ \gamma_2J^{(2)}(x^0)))(x^0-x^*) +\tilde R(\varepsilon)\\
&\le \|x^0-x^*\|^2-2\varepsilon(\beta_{11}\|x^0-x^*\|^{m}\\
&+\beta_{21}\|x^0-x^*\|^{m-2}) +\tilde R(\varepsilon), \, \tilde R(\varepsilon)=o(\varepsilon^{5/4})\text{ as }\varepsilon\to 0. 
\end{aligned}
$$
A more detailed proof will be given in the extended version of the paper.
\begin{figure}[!b]
    \centering
\includegraphics[width=1\linewidth]{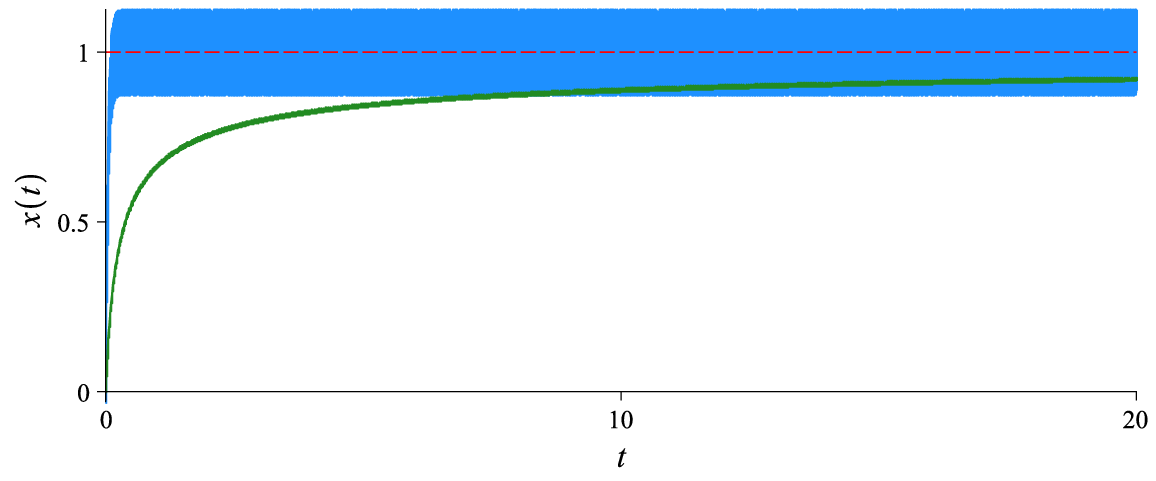}
    \caption{Blue: the proposed ES system in \eqref{ex1_we}, which converges practically in an exponential rate to 1, the minimum point of $J_1(x)=(x-1)^4$; green: a traditional ES approach from literature \eqref{ex1_Durr}, which converges practically to 1 in a polynomial rate; red: the minimum point of $J_1$, $x^*=1$.}
   \label{fig_pol4}
\end{figure}
\section{Numerical simulations}
To demonstrate the effectiveness of the proposed approach, we take a cost function as the fourth-order polynomial,
$
J(x)=(x-1)^4,
$
and apply first-order-based approach~\cite{DurrAuto},
\begin{equation}
    \label{ex1_Durr}
\dot x=2\sqrt{\pi\varepsilon^{-1}}\left(J_1(x)\cos{2\pi t}{\varepsilon}^{-1}+\sin{2\pi t}{\varepsilon}^{-1}\right),
\end{equation}
and fourth-order-based approach with two inputs defined as in Statement~1:
\begin{equation}
    \label{ex1_we}
\dot x{=}2\left({2\pi}{\varepsilon}^{-1}\right)^{3/4}\left(3J_1(x)\sin{6\pi t}{\varepsilon}^{-1}{+}\cos{2\pi t}{\varepsilon}^{-1}\right).
\end{equation}
We initiate both equations at $x(0)=0$ and put $\varepsilon=10^{-4}$. The results of the numerical simulations are shown in Fig.\ref{fig_pol4}, demonstrating a significantly faster convergence rate for the solutions of system~\eqref{ex1_we} compared to~\eqref{ex1_Durr}, while also exhibiting a higher amplitude of oscillations. In our future studies, we plan to explore alternative choices for the functions  $g_1$  and  $g_2$, especially,  with vanishing amplitudes similarly to as proposed in~\cite{GZE18}, which have been shown to improve the qualitative behavior of the solutions. Another promising approach for reducing oscillations involves the use of time-varying gain techniques, such as in~\cite{GE24_CDC,Pokh23,yilmaz2025exponential}. 




\section{Conclusions \& Future Work}

In this paper, we have introduced a novel framework for extremum seeking  control that utilizes higher-order Lie bracket approximations to achieve improved convergence properties, particularly for cost functions with polynomial-like behavior near their minima. Unlike many ES approaches that rely on first-order Lie bracket systems and yield exponential convergence only for quadratic-like cost functions, the proposed approach may ensure the exponential convergence even when the cost function behaves locally like $\|x - x^*\|^m$ with $m > 2$. 
This is achieved by integrating ideas from differential geometric control theory related to higher-order Lie bracket averaging for control-affine systems~\cite{PE23} with control design techniques for high-order nonholonomic systems~\cite{GZ24_CDC,Gau14}, and advanced analytical tools for studying dynamical systems, in particular, the Chen--Fliess series expansion \cite{GZE18,GZ23}.

Let us note that, to better introduce our approach and due to space limitations, we have considered only the case of a single-variable cost function. In our future work, we plan to extend this approach to multi-variable cost functions under less restrictive assumptions on their local behavior, as well as to general extremum seeking problem statement like in~\cite{Durr17,GE21}. Another important research direction concerns exploring the possible choices of generating vector fields in Lemma~3.       We expect that, as in~\cite{GZE18}, it is possible to achieve asymptotic (in particular, exponential) stability in the sense of Lyapunov by appropriately choosing vector fields that vanish at the extremum. In such scenarios, an improved convergence rate would be even more beneficial, as the amplitude of oscillations would also vanish. 

All in all, we believe that this paper opens new possibilities for designing extremum seeking control laws with enhanced flexibility, and  initiates a promising direction for further developments of high-order Lie bracket methods in optimization tasks. 



 \bibliographystyle{ieeetr}
\bibliography{biblio_ES,references}

\begin{thebibliography}{10}

\bibitem{krstic2008extremum}
M.~Krstic and J.~Cochran, ``Extremum seeking for motion optimization: From bacteria to nonholonomic vehicles,'' in {\em 2008 Chinese Control and Decision Conference}, pp.~18--27, IEEE, 2008.

\bibitem{grushkovskaya2018family}
V.~Grushkovskaya, S.~Michalowsky, A.~Zuyev, M.~May, and C.~Ebenbauer, ``A family of extremum seeking laws for a unicycle model with a moving target: theoretical and experimental studies,'' in {\em 2018 European Control Conference (ECC)}, pp.~1--6, IEEE, 2018.

\bibitem{ECC_2024}
S.~Bajpai, A.~A. Elgohary, and S.~A. Eisa, ``Model-free source seeking by a novel single-integrator with attenuating oscillations and better convergence rate: Robotic experiments,'' in {\em 2024 European Control Conference (ECC)}, pp.~472--479, IEEE, 2024.

\bibitem{eisa2023}
S.~A. Eisa and S.~Pokhrel, ``Analyzing and mimicking the optimized flight physics of soaring birds: A differential geometric control and extremum seeking system approach with real time implementation,'' {\em SIAM Journal on Applied Mathematics}, pp.~S82--S104, 2023.

\bibitem{kurzweil1987limit}
J.~Kurzweil and J.~Jarn{\'\i}k, ``Limit processes in ordinary differential equations,'' {\em Zeitschrift f{\"u}r angewandte Mathematik und Physik ZAMP}, vol.~38, pp.~241--256, 1987.

\bibitem{Sus91}
H.~J. Sussmann and W.~Liu, ``Limits of highly oscillatory controls and the approximation of general paths by admissible trajectories,'' in {\em Proc. 30th IEEE CDC}, pp.~437--442, 1991.

\bibitem{bullo04}
F.~Bullo and A.~D. Lewis, {\em Geometric Control of Mechanical Systems}, vol.~49 of {\em Texts in Applied Mathematics}.
\newblock New York-Heidelberg-Berlin: Springer Verlag, 2004.

\bibitem{Liu97}
W.~Liu, ``An approximation algorithm for nonholonomic systems,'' {\em SIAM J. Control Optim.}, vol.~35, no.~4, pp.~1328--1365, 1997.

\bibitem{DurrAuto}
H.-B. D\"{u}rr, M.~S. Stankovi\'{c}, C.~Ebenbauer, and K.~Johansson, ``Lie bracket approximation of extremum seeking systems,'' {\em Automatica}, vol.~49, pp.~1538--1552, 2013.

\bibitem{Sch14}
A.~Scheinker and M.~Krsti\'{c}, ``Extremum seeking with bounded update rates,'' {\em Systems \& Control Letters}, vol.~63, pp.~25--–31, 2014.

\bibitem{GZE18}
V.~Grushkovskaya, A.~Zuyev, and C.~Ebenbauer, ``On a class of generating vector fields for the extremum seeking problem: {L}ie bracket approximation and stability properties,'' {\em Automatica}, vol.~94, pp.~151--160, 2018.

\bibitem{Labar19}
C.~Labar, E.~Garone, M.~Kinnaert, and C.~Ebenbauer, ``Newton-based extremum seeking: A second-order {L}ie bracket approximation approach,'' {\em Automatica}, vol.~105, pp.~356--367, 2019.

\bibitem{PE23}
S.~Pokhrel and S.~A. Eisa, ``Higher-order {L}ie bracket approximation and averaging of control-affine systems with application to extremum seeking,'' {\em arXiv preprint arXiv:2310.07092}, 2023.

\bibitem{suttner2020extremum}
R.~Suttner, ``Extremum seeking control for a class of nonholonomic systems,'' {\em SIAM Journal on Control and Optimization}, vol.~58, no.~4, pp.~2588--2615, 2020.

\bibitem{ZuSIAM}
A.~Zuyev, ``Exponential stabilization of nonholonomic systems by means of oscillating controls,'' {\em SIAM J Control Optim}, vol.~54, no.~3, pp.~1678--1696, 2016.

\bibitem{doi:10.1080/00207179.2016.1257157}
A.~Zuyev and V.~Grushkovskaya, ``Motion planning for control-affine systems satisfying low-order controllability conditions,'' {\em International Journal of Control}, vol.~90, no.~11, pp.~2517--2537, 2017.

\bibitem{GZ23}
V.~Grushkovskaya and A.~Zuyev, ``Motion planning and stabilization of nonholonomic systems using gradient flow approximations,'' {\em Nonlinear Dynamics}, vol.~111, no.~23, pp.~21647--21671, 2023.

\bibitem{GZ24_CDC}
V.~Grushkovskaya and A.~Zuyev, ``Design of stabilizing feedback controllers for high-order nonholonomic systems,'' {\em IEEE Control Systems Letters}, 2024.

\bibitem{Ar03}
K.~Ariyur and M.~Krstic, {\em Real-time optimization by extremum-seeking control}.
\newblock John Wiley{\&}Sons, 2003.

\bibitem{scheinker2024100}
A.~Scheinker, ``100 years of extremum seeking: A survey,'' {\em Automatica}, vol.~161, p.~111481, 2024.

\bibitem{Kr00}
M.~Krsti{\'c} and H.-H. Wang, ``Stability of extremum seeking feedback for general nonlinear dynamic systems,'' {\em Automatica}, vol.~36, no.~4, pp.~595--601, 2000.

\bibitem{yilmaz2025exponential}
C.~T. Yilmaz, M.~Diagne, and M.~Krstic, ``Exponential and prescribed-time extremum seeking with unbiased convergence,'' {\em Automatica}, vol.~179, p.~112392, 2025.

\bibitem{guay2015time}
M.~Guay and D.~Dochain, ``A time-varying extremum-seeking control approach,'' {\em Automatica}, vol.~51, pp.~356--363, 2015.

\bibitem{khalil2002nonlinear}
H.~K. Khalil, ``Nonlinear systems third edition,'' {\em Prentice Hall}, vol.~115, 2002.

\bibitem{Maggia2020higherOrderAvg}
M.~Maggia, S.~A. Eisa, and H.~E. Taha, ``On higher-order averaging of time-periodic systems: reconciliation of two averaging techniques,'' {\em Nonlinear Dynamics}, vol.~99, pp.~813--836, Jan 2020.

\bibitem{GZ13}
V.~Grushkovskaya and A.~Zuyev, ``Asymptotic behavior of solutions of a nonlinear system in the critical case of q pairs of purely imaginary eigenvalues,'' {\em Nonlinear Analysis: Theory, Methods \& Applications}, vol.~80, pp.~156--178, 2013.

\bibitem{G16}
V.~Grushkovskaya, ``On the influence of resonances on the asymptotic behavior of trajectories of nonlinear systems in critical cases,'' {\em Nonlinear Dynamics}, vol.~86, no.~1, pp.~587--603, 2016.

\bibitem{La95}
F.~Lamnabhi-Lagarrigue, ``{V}olterra and {F}liess series expansions for nonlinear systems,'' in {\em The Control Handbook} (W.~S. Levine, ed.), pp.~879--888, CRC Press, 1995.

\bibitem{Gau14}
J.-P. Gauthier and M.~Kawski, ``Minimal complexity sinusoidal controls for path planning,'' {\em Proc. 53rd IEEE CDC}, pp.~3731--3736, 2014.

\bibitem{GE24_CDC}
V.~Grushkovskaya and C.~Ebenbauer, ``Step-size rules for {L}ie bracket-based extremum seeking with asymptotic convergence guarantees,'' {\em IEEE Control Systems Letters}, 2024.

\bibitem{Pokh23}
S.~Pokhrel and S.~A. Eisa, ``Control-affine extremum seeking control with attenuating oscillations: A {L}ie bracket estimation approach,'' in {\em 2023 Proceedings of the Conference on Control and its Applications (CT)}, pp.~133--140, SIAM, 2023.

\bibitem{Durr17}
H.-B. D{\"u}rr, M.~Krsti{\'c}, A.~Scheinker, and C.~Ebenbauer, ``Extremum seeking for dynamic maps using {L}ie brackets and singular perturbations,'' {\em Automatica}, vol.~83, pp.~91--99, 2017.

\bibitem{GE21}
V.~Grushkovskaya and C.~Ebenbauer, ``Extremum seeking control of nonlinear dynamic systems using {L}ie bracket approximations,'' {\em International Journal of Adaptive Control and Signal Processing}, vol.~35, no.~7, pp.~1233--1255, 2021.

\end{thebibliography}

\end{document}